\documentclass[reqno, 12pt]{amsart}
\usepackage{graphicx,amssymb}
\newtheorem{thm}{Theorem}[section]

\theoremstyle{definition}
\newtheorem{defn}[thm]{Definition}
\theoremstyle{remark}
\newtheorem{rem}[thm]{Remark}

\numberwithin{equation}{section}
\def\phi{\varphi}
\def\theta{\vartheta}

\allowdisplaybreaks
\begin{document}
%\selectlanguage{english}
\title[On Cayley Identity\dots]{On Cayley Identity for Self-Adjoint Operators in Hilbert Spaces}
\author{Alexander V. Kiselev}
\author{Serguei N. Naboko}
\address{Department of Higher Mathematics and Mathematical Physics,
St.Petersburg State University, 1 Ulianovskaya Street,
St.Petersburg, St. Peterhoff 198504 Russia}
\address{Department of Higher Mathematics and Mathematical Physics,
St.Petersburg State University, 1 Ulianovskaya Street,
St.Petersburg, St. Peterhoff 198504 Russia}
\email{alexander.v.kiselev@gmail.com} \subjclass{Primary 47A10;
Secondary 47A55}
\thanks{The first author gratefully acknowledges support received from the EPSRC grant EP/D00022X/2. The first and second authors acknowledge
partial support from the RFBR grants 09-01-00515-a and
11-01-90402-Ukr\_f\_a.}
%%% ----------------------------------------------------------------------
\begin{abstract}
We prove an analogue to the Cayley identity for an arbitrary self-adjoint operator in a Hilbert space. We also provide two new
ways to characterize vectors belonging to the singular spectral subspace in terms of the analytic properties of the
resolvent of the operator, computed on these vectors. The latter are analogous to those used routinely in the scattering theory for the
absolutely continuous subspace.
\end{abstract}
%%% ----------------------------------------------------------------------
\maketitle

\specialsection{Introduction}
If $M$ is a matrix in $\mathbb C^n$ and $d_M(\lambda):=\det(M-\lambda)$ is its characteristic polynomial, the
celebrated Cayley identity says that
$$
d_M(M)\equiv 0.
$$
In \cite{KiselevNaboko} we have studied the ``almost Hermitian'' spectral subspace of a nonself-adjoint, non-dissipative operator $L$. The following
criterion has been established: a nonself-adjoint operator possesses almost Hermitian spectrum (i.e., its almost Hermitian spectral subspace
coincides with the Hilbert space $H$) iff a natural generalization of Cayley identity hold both for the operator itself and its adjoint.

This generalization of Cayley identity is formulated in terms of the so-called weak outer annihilation. The following definition of it has been
suggested:
\begin{defn}\label{oldone}
Let $\gamma(\lambda)$ be an outer \cite{Hoffman} in the upper half-plane $\mathbb
C_+$ uniformly bounded scalar analytic function. We call this function a
\emph{weak annihilator} of an operator $L$, if
\begin{equation}\label{definition1}
w-\lim_{\varepsilon\downarrow 0} \gamma(L+i\varepsilon) =0.
\end{equation}
\end{defn}
As a by-product of the aforementioned analysis of nonself-adjoint operators and using essentially nonself-adjoint techniques (i.e., the dilation
of a dissipative operator, see \cite{Nagy}) we have further been able to prove, that a self-adjoint operator $A$ has trivial absolutely continuous subspace
if and only if $A$ is weakly annihilated in the sense of the above definition.

Moreover, the corresponding outer analytic function admits an explicit choice, i.e., it can be chosen to be equal to
the perturbation determinant $D_{A/A-iV}(\lambda)$ of the pair $A$, $A-iV$ \cite{Krein}, where $V$ is an auxiliary non-negative
trace class operator.

The natural question on the possibility to formulate a ``local''
version of this criterion for self-adjoint operators with mixed
spectrum, i.e., how to ascertain in similar terms whether the
spectrum of a given self-adjoint operator $A$ is purely singular
inside some Borel set of the real line, was posed some time ago by
Prof. David Pearson. The present paper is an attempt to give an
(in our view, so far incomplete) answer to this question. We prove
the result envisaged in two quite different flavours: the one that
we prefer (but analytic difficulties then only allow us to give
the proof under rather restrictive assumptions on the operators'
spectrum, see below) and the one that actually allows to give a
rigorous proof in the most general case.

The paper is organized as follows.

Since the functional model of a nonself-adjoint operator is of crucial importance for our approach and
the proof of our main result relies heavily upon the symmetric form of the Nagy-Foia\c s
functional model due to Pavlov \cite{pavlov2,pavlov} (see also the paper \cite{MIAN} by Naboko),
we continue with a brief introduction to the main concepts and results obtained in this area in Section 2.

Section 3 contains our main result, which may be viewed as a
generalization of the Cayley identity to self-adjoint operators
with an arbitrary spectral structure.

Finally, in Section 4 we derive two new characterizations of vectors belonging to the singular spectral
subspace of a self-adjoint operator in terms of the analytic properties of the
resolvent of the operator, computed on these vectors. The latter are analogous to those used routinely in the scattering theory for the
absolutely continuous subspace.

At this time, we have elected to postpone the discussion of
possible applications, since the meaningful examples we have in
mind, i.e., the examples in which the singularity of the spectrum
in a given set is either unknown or cannot be obtained by some
simpler classical techniques, do require substantial an
non-trivial analysis to be fully considered.

\specialsection{The functional model of a dissipative operator} In the present section we
briefly recall the functional model of a nonself-adjoint operator
constructed in \cite{Nagy,pavlov} in the dissipative case and then
extended in \cite{NabokoI,NabokoII,MIAN,RyzhovPHD} to the case of
a wide class of non-dissipative operators. We consider a class of
nonself-adjoint operators of the form \cite{MIAN} $ L=A+iV,$ where
$A$ is a self-adjoint operator in $H$ defined on the domain $D(A)$
and the perturbation $V$ admits the factorization $V={\alpha J
\alpha}/{2}$, where $\alpha$ is a non-negative self-adjoint
operator in $H$ and $J$ is a unitary operator in an auxiliary
Hilbert space $E$, defined as the closed range of the operator
$\alpha$: $E\equiv \overline{R(\alpha)}$. This factorization
corresponds to the polar decomposition of the operator $V$. It can
also be easily generalized to the ``node'' case \cite{VeselovPHD},
where $J$ acts in an auxiliary Hilbert space $\mathfrak H$ and
$V=\alpha^* J \alpha/2$, $\alpha$ being an operator acting from
$H$ to $\mathfrak H$. In order that the expression $A+iV$ be
meaningful, we impose the condition that $V$ be $(A)$-bounded with
relative bound less than 1, i.~e., $D(A)\subset D(V)$ and for some
$a$ and $b$ ($a<1$) the condition $ \|Vu\|\leq
a\|Au\|+b\|u\|,\quad u\in D(A) $ is satisfied, see \cite{Kato}.
Then the operator $L$ is well-defined on the domain $D(L)=D(A)$.

Alongside with the operator $L$ we are going to consider the
maximal dissipative operator $L^{\|}=A+i\frac{\alpha^2}2$ and the
one adjoint to it, $L^{-\|}\equiv L^{\|*}=A-i\frac{\alpha^2}2$.
Since the functional model for the dissipative operator $L^{\|}$
will be used below, we require that $L^{\|}$ is completely
nonself-adjoint, i.~e., that it has no reducing self-adjoint
parts. This requirement is not restrictive in our case due to
Proposition 1 in \cite{MIAN}.

We also note that the functional model in the general case of
operators with not necessarily additive imaginary part and with
non-empty resolvent set has been developed in \cite{RyzhovPHD}.

Now we are going to briefly describe a construction of the
self-adjoint dilation of the completely nonself-adjoint
dissipative operator $L^{\|}$, following \cite{Nagy,pavlov}, see
also \cite{MIAN}.

The characteristic function $S(\lambda)$ of the operator $L^{\|}$
is a contractive, analytic operator-valued function acting in the
Hilbert space $E$, defined for $Im \lambda>0$ by
\begin{equation}\label{sdef}
S(\lambda)=I+i\alpha(L^{-\|}-\lambda)^{-1}\alpha.
\end{equation}
In the case of an unbounded $\alpha$ the characteristic function
is first defined by the latter expression on the manifold $E\cap
D(\alpha)$ and then extended by continuity to the whole space $E$.
The definition given above makes it possible to consider
$S(\lambda)$ for $Im \lambda<0$ with
$S(\overline{\lambda})=(S^*(\lambda))^{-1}$ provided that the
inverse exists at the point $\lambda$. Finally, $S(\lambda)$
possesses  boundary values on the real axis in the strong topology
sense: $S(k)\equiv S(k+i0),\ k\in \mathbb{R}$ (see \cite{Nagy}).

Consider the model space $\mathcal{H}=L_2(
\begin{smallmatrix}I&S^*\\
S&I\\
\end{smallmatrix}),$ which is defined in \cite{pavlov} (see also \cite{Nik1} for description of
general coordinate-free models) as Hilbert space of two-component
vector-functions $(\tilde g,g)$ on the axis ($\tilde g(k),g(k) \in
E, k\in\mathbb{R}$) with metric
$$
\left\langle\binom{\tilde g}{g},\binom{\tilde g}{g}\right\rangle=
\int_{-\infty}^{\infty}\left\langle
\begin{pmatrix}I& S^*(k)\\S(k)&I\\
\end{pmatrix}\binom{\tilde g(k)}{g(k)},\binom{\tilde g(k)}{g(k)}
\right\rangle_{E\oplus E}dk.
$$
It is assumed here that the set of two-component functions has
been factored by the set of elements with norm equal to zero.
Although we consider $(\tilde g,g)$ as a symbol only, the formal
expressions $g_-:=(\tilde g + S^{*} g)$ and $g_+:=(S\tilde g +g)$
(the motivation for the choice of notation is self-evident from
what follows) can be shown to represent some true
$L_2(E)$-functions on the real line. In what follows we plan to
deal mostly with these functions.

Define the following orthogonal subspaces in $\mathcal{H}:$
$$
D_-\equiv \binom{0}{H_-^2(E)},\ D_+\equiv \binom{H_+^2(E)}{0},\
K\equiv \mathcal{H}\ominus(D_-\oplus D_+),
$$
where $H^2_{+(-)}(E)$ denotes the Hardy class \cite{Nagy} of
analytic functions $f$ in the upper (lower) half-plane taking
values in the Hilbert space $E$. These subspaces are ``incoming''
and ``outgoing'' subspaces, respectively, in the sense of
\cite{Lax}.

The subspace $K$ can be described as $K=\{(\tilde g,g)\in
\mathcal{H}:\ g_-\equiv\tilde g+S^*g\in H^2_-(E),g_+\equiv S\tilde
g+g\in H^2_+(E)\}.$ Let $P_K$ be the orthogonal projection of the
space $\mathcal{H}$ onto $K$, then
$$
P_K\binom{\tilde g}{g}=\binom{\tilde g - P_+(\tilde g+S^* g)}
{g-P_-(S\tilde g +g)},
$$
where $P_{\pm}$ are the orthogonal Riesz projections of the space
$L_2(E)$ onto $H^2_{\pm} (E).$

The following Theorem holds \cite{Nagy,pavlov}:
\begin{thm}
The operator $(L^{\|}-\lambda_0)^{-1}$ is unitarily equivalent to
the operator $P_K (k-\lambda_0)^{-1}|_K$ in the space $K$ for all
$\lambda_0, Im \lambda_0<0$.
\end{thm}
This means, that the operator of multiplication by $k$ in
$\mathcal H$ serves as a minimal ($clos_{Im \lambda\not=0}
(k-\lambda)^{-1}K=\mathcal{H}$) self-adjoint dilation \cite{Nagy}
of the operator $L^{\|}.$

\specialsection{Characterization of singular spectrum in terms of
weak annihilation}

In the present section, we attempt to provide a localized
criterion of pure singularity of the spectrum of a general
self-adjoint operator inside a given set of the real line,
building upon the technique and approach developed in
\cite{KiselevNaboko}. It is worth mentioning that not only the
proofs of our results in this direction exploit essentially
nonself-adjoint (in particular, functional model related)
techniques, but even certain crucial objects of the
nonself-adjoint spectral theory appear already in their
statements.

Our next Theorem in our view constitutes the most natural
localization of the corresponding ``global'' result of
\cite{KiselevNaboko}. Unfortunately, we are only able to prove the
result in this natural form in the case when both ends of the
interval $\Delta$ where one wants to ascertain pure singularity of
the spectrum are located inside a spectral gap. Any attempt to get
rid of this rather horrible restriction requiring crucial a-priori
information on the spectral structure fails due to the lack of
control over the annihilating function at the endpoints of the
interval $\Delta$. It seems that in the general setting one has to
resort to a quite different (and less natural) definition of
annihilation (see Theorem \ref{sacase2} below).

\begin{thm}\label{sacase}
Let $A$ be a (possibly, unbounded) self-adjoint operator in the
Hilbert space $H$. Let a point $\lambda_0\in\mathbb R$ belong together with some neighborhood $\Delta$ to the resolvent set of $A$.
Then the following two statements are
equivalent.
\begin{itemize}
\item[(i)] The spectrum of $A$ to the left of the point $\lambda_0$ is purely singular;
\item[(ii)]
There exists an outer bounded in the upper half-plane non-trivial (i.e., non-constant) scalar
function $\gamma(\lambda)$ with real boundary values almost everywhere on $(\lambda_0,+\infty)$ and non-real boundary values
almost everywhere on $(-\infty,\lambda_0)$, weakly annihilating the operator
$A$, i.e.,
$$
w-\lim_{\varepsilon\downarrow 0} (\gamma(A+i\varepsilon)-\gamma_*(A-i\varepsilon))=0,
$$
where $\gamma_*(\lambda):=\bar\gamma(\bar\lambda)$ is an outer bounded in the lower half-plane analytic function.
\end{itemize}
\end{thm}

\begin{proof}
Choose $V$ to be a trace class non-negative self-adjoint operator
in the Hilbert space $H$ such that
\begin{equation}\label{maximalityA}
\bigvee_{Im\ \lambda\not=0} (A-\lambda)^{-1}VH=H.
\end{equation}
Clearly, such choice is always possible.

We follow the approach developed in \cite{MIAN} for the operators
$L$ admitting the representation
$L_\kappa=A+\alpha\kappa\alpha/2$, where $\alpha\geq 0$ is a
non-negative operator in the Hilbert space $H$ and $\kappa$ is a
bounded operator in the subspace $E$, being the closure of the
range of $\alpha$. Choose $\alpha$ to be a Hilbert-Schmidt class
operator defined by the formula $\alpha =\sqrt{2V}\in \mathfrak
S_2$. Then the operator $L_\kappa$ is well-defined on the domain
$D(L_\kappa)=D(A)$. Moreover, $L_\kappa\equiv A$ when $\kappa=0$,
i.e., $L_0\equiv A$. Consider the dissipative operator $L^\|\equiv
A+iV$ (this operator coincides with $L_{iI}$). Clearly, it is a
maximal dissipative operator in $H$; moreover, it is easy to see
that the condition \eqref{maximalityA} guarantees that it is also
completely nonself-adjoint.

Construct the functional model based on the operator $L^\|$ (see
Section 1 above). In the corresponding dilation space $\mathcal H$
the following formulae describe the action of the resolvent
$(A-\lambda)^{-1}$ on all vectors $(\tilde g,g)\in K$, as above
$K$ being the model image of $H$ (see \cite{MIAN}):
\begin{gather}\label{samodel}
(A-\lambda)^{-1}\binom{\tilde g}{g}= P_K \frac 1{(k-\lambda)}
\binom{\tilde g} {g-\frac 1 2 \left( I+
(S^*(\overline\lambda)-I)\frac 1 2\right)^{-1}g_-(\lambda)},\quad
Im\
\lambda<0\\
(A-\lambda)^{-1}\binom{\tilde g}{g}= P_K \frac 1{(k-\lambda)}
\binom{\tilde g- \frac 1 2 \left( I+ (S(\lambda)-I)\frac 1
2\right)^{-1} g_+(\lambda)} {g},\quad Im\ \lambda>0.
\end{gather}
Here $S(\lambda)$ is the characteristic function of completely
nonself-adjoint maximal dissipative operator $L^\|$, all the other
notation has already been introduced above.

We introduce the following notation for the operator functions,
appearing in this representation: $\Theta_A(\lambda):=I+
(S(\lambda)-I)\frac 1 2$ and $\Theta_A'(\lambda):=I+
(S^*(\overline\lambda)-I)\frac 1 2$. The functions $\Theta_A$ and
$\Theta_A'$ are bounded analytic operator functions in half-planes
$\mathbb C_+$ and $\mathbb C_-$, respectively.

Recall that the characteristic function $S(\lambda)$ is a
contraction in the upper half plane. It follows that, since
$\Theta_A(\lambda)=(I+S(\lambda))/2$ and $\Theta_A'(\lambda)=
(I+S^*(\overline\lambda))/2$, they are outer contractions (see
\cite{Nagy}) in the half-planes $\mathbb C_+$ and $\mathbb C_-$,
respectively.

By definition of $S(\lambda)$, both operator functions also have
well-defined outer \cite{NabokoVeselov,Nagy} determinants\footnote{%
It is easy to see that $\gamma_A(\lambda)$ in fact coincides with
the perturbation determinant $D_{A/A-iV}(\lambda)$ of the pair
$A$, $A-iV$ \cite{Krein}}
$\gamma_A$, $\gamma_A'$, bounded (in fact, contractive) in their
respective half-planes. It is also clear that
$\gamma_A'(\lambda)=\overline{\gamma_A(\overline\lambda)}$. We
remark, that $\gamma_A(\lambda)$ is a clearly non-zero function since $\lim_{\tau\to+\infty}
\gamma_A(i\tau)=1$.

W.l.o.g. assume, that the point $\lambda=0$ together with its
neighborhood $\Delta$ belongs to the resolvent set of the operator $A$. It
follows that since the operator $L^{\|}$ is completely
nonself-adjoint and dissipative, the same neighborhood also
belongs to its resolvent set. Thus the function $S(\lambda)$
admits analytic continuation to $\mathbb C_-$ through the named
neighborhood of zero and the determinant $\gamma_A(\lambda)$ is
$C^\infty$ there.

Since $\gamma(\lambda)$ is an outer function, it admits the
following representation in terms of the logarithm of its boundary
values on the real line:
$$
\gamma_A(\lambda)=e^{ic}\exp\left\lbrace\frac i\pi \int_{\mathbb
R}\left(\frac 1{\lambda -t} +\frac t{1+t^2} \log |\gamma(t)| dt \right)\right\rbrace,
$$
where $\gamma(t):=\gamma(t+i0)$ are the boundary values of the function $\gamma$ from above and $c$ is some real constant. From \eqref{samodel} it further
follows, that $\gamma_A (t)$ is separated from zero on $\Delta$.

Fix a point $\delta_0$ such that $-\delta_0\in \Delta$ and let
$$
\gamma_1(\lambda)=e^{ic}\exp\left\lbrace\frac i\pi \int_{-\infty}^{-\delta_0}\left(\frac 1{\lambda -t} +\frac t{1+t^2}\right) \log |\gamma_A(t)| dt\right\rbrace
$$
be the new outer (by construction), bounded in the upper half-plane function. Let further $\phi(t)$ be the
harmonic conjugate (or, in other words, the Hilbert transform) of the
function $f(t)$, equal to $\log|\gamma_A(t)|$ on $(-\infty,-\delta_0)$ and to 0 elsewhere. Clearly, the function $\phi(t)$
is itself infinitely smooth on any interval $[-\delta_1, +\infty)$ provided that $\delta_1<\delta_0$.

Choose yet another function $\gamma_2(\lambda)$ as follows:
$$
\gamma_2(\lambda)=e^{-ic}\exp\left\lbrace\frac 1\pi \int_{-\infty}^{+\infty}\left(\frac 1{\lambda -t} +\frac t{1+t^2}\right) \phi_1(t) dt\right\rbrace,
$$
where $\phi_1(t)$ is any $C^1(\mathbb R)$ function such that $\phi_1(t)\equiv\phi(t),$ $t\geq 0$.
As it is easily seen, on the right half-line $\arg \gamma_2(t+i0)=-\arg \gamma_1(t+i0)$ almost everywhere. What's more, since $\phi_1(t)$ is $C^1$ on the
real line, its harmonic conjugate is continuous and thus bounded \cite{Koosis}. It follows, that the function $\gamma_2(\lambda)$ is itself bounded and outer
in the upper half-plane, admitting the following representation:
$$
\gamma_2(\lambda)=e^{-ic}\exp\left\lbrace\frac i\pi \int_{-\infty}^{+\infty}\left(\frac 1{\lambda -t} +\frac t{1+t^2}\right) (-\widetilde\phi_1(t)) dt\right\rbrace,
$$
where $\widetilde\phi_1(t)$ is the harmonic conjugate of the function $\phi_1$.

Consider the function $\gamma(\lambda):=\gamma_1(\lambda)\gamma_2(\lambda)$. By virtue of its construction, it is a bounded outer function in the upper
half-plane with almost everywhere real boundary values on the right half-line. What's more, together with its first factor it cancels out the zeroes
of the function $\gamma_A(\lambda)$ on the interval $(-\infty,-\delta_0)$: $|\gamma(t+i\varepsilon)/\gamma_A(t+i\varepsilon)|\leq C$ uniformly in $\varepsilon$
for some finite
constant $C$ and every $t\leq -\delta_0$.

It remains to be seen that this function can be chosen in a way
such that its boundary values to the left of the point $0$ are
non-real almost everywhere. In fact, this can be safely assumed
w.l.o.g.: if not, denote by $\Omega\subset (-\infty,0)$ the set of
points where the corresponding boundary values are real almost
everywhere. Then consider a non-negative smooth enough function
$g(k)$ having its support equal to the closure of $\Omega$ and
define $\hat g(\lambda):=\int \frac{g(k)}{k-\lambda}dk$.
Multiplication by this outer bounded factor clearly equips the
function $\gamma(\lambda)$ with the properties required by the
Theorem.

We will now prove that $\gamma(\lambda)$ weakly annihilates the
self-adjoint operator $A$ in the sense of the Theorem.

First, let $u$ belong to the spectral subspace $E_A(0,+\infty)H$
of the operator $A$, where $E_A(\cdot)$ is the operator-valued
spectral measure associated with $A$. Then by the spectral theorem
and by Lebesgue dominated convergence theorem it is easy to see
that
$$
\lim_{\varepsilon\downarrow 0} \langle(\gamma(A+i\varepsilon)-\gamma_*(A-i\varepsilon))u,v\rangle=\int_0^{+\infty} (\gamma(k+i0)-\bar\gamma(k+i0) d\mu_{u,v} (k)=0
$$
for all $v$ in $H$ (in a nutshell, we have used the fact that the function $\gamma_*$ by its construction is an analytic continuation of the function $\gamma$ to the
lower half-plane).

It remains to be seen that if $u=E_A(-\infty,0)H$, then
$$
\lim_{\varepsilon\downarrow 0} \langle\gamma(A+i\varepsilon)u,v\rangle=0
$$
and
$$
\lim_{\varepsilon\downarrow 0} \langle\gamma_*(A-i\varepsilon)u,v\rangle=0
$$
for all $v$ in $H$, provided that the spectrum of the operator $A$ is purely singular to the left of the point zero. We will check the first identity above,
the second being verified analogously.

The bounded (due to v. Neumann inequality \cite{Nagy} or,
alternatively, due to the spectral theorem) operator
$\gamma(A+i\varepsilon)$ is defined by the Riesz-Dunford
integral,
\begin{multline*}
\langle \gamma(A+i\varepsilon) u, v\rangle= \frac 1{2\pi i}
\left(\int_{-\infty+3i\varepsilon/2}^{-\delta_0+3i\varepsilon/2}-\int_{-\infty+i\varepsilon/2}^{-\delta_0+i\varepsilon/2}\right)
\gamma_A(\lambda) \left\langle (A+i\varepsilon-\lambda)^{-1}
u,v\right\rangle d\lambda.
\end{multline*}

Using the model representation \eqref{samodel} we then immediately
obtain:
\begin{multline}\label{mainforexistenceA}
\left\langle \gamma(A+i\varepsilon)\begin{pmatrix} \tilde g \\ g
\end{pmatrix},\begin{pmatrix}
\tilde f \\ f
\end{pmatrix}\right\rangle = \left\langle \gamma(k+i\varepsilon) \begin{pmatrix}
\tilde g \\ g
\end{pmatrix},
\begin{pmatrix}
\tilde f \\ f
\end{pmatrix}\right\rangle+\\
\frac 1{2\pi i} \int_{-\infty}^{-\delta_0}
\gamma(t+i\frac\varepsilon2)\left \langle \frac 1
{k-(t-i\frac\varepsilon2)} \begin{pmatrix}
 0 \\
\frac 1 2\Theta_A^{'\ -1}(t-i\frac\varepsilon2)
g_-(t-i\frac\varepsilon2)
\end{pmatrix}, \begin{pmatrix}
\tilde f \\ f
\end{pmatrix}\right\rangle dt-\\
\frac 1{2\pi i} \int_{-\infty}^{-\delta_0}
\gamma(t+i\frac{3\varepsilon}2)\left \langle \frac 1
{k-(t+i\frac\varepsilon2)} \begin{pmatrix}
\frac 1 2\Theta_A^{-1}(t+i\frac\varepsilon2) g_+(t+i\frac\varepsilon2) \\
0
\end{pmatrix}, \begin{pmatrix}
\tilde f \\ f
\end{pmatrix}\right\rangle dt.
\end{multline}

Rewriting
$\Theta'_A(\lambda)=\Omega'(\lambda)/\bar\gamma_A(\bar\lambda)$
and $\Theta_A(\lambda)=\Omega(\lambda)/\gamma_A(\lambda)$ with
bounded in the lower (resp., upper) half-plane operator function
$\Omega'$ (resp., $\Omega$), it is now easy to see that the last
expression assumes the following form:
\begin{multline}\label{existence2}
\left\langle \gamma(A+i\varepsilon)\begin{pmatrix}
\tilde g \\ g
\end{pmatrix},\begin{pmatrix}
\tilde f \\ f
\end{pmatrix}\right\rangle = \left\langle \gamma(k+i\varepsilon) \begin{pmatrix}
\tilde g \\ g
\end{pmatrix},
\begin{pmatrix}
\tilde f \\ f
\end{pmatrix}\right\rangle+\\
\int_{-\infty}^{-\delta_0}
\frac{\gamma(t+i\frac\varepsilon2)}{\overline{\gamma_A(t+i\frac\varepsilon2)}}
\left \langle
\frac 1 2\Omega'(t-i\frac\varepsilon2) g_-(t-i\frac\varepsilon2), f_+(t+i\frac\varepsilon2)
\right\rangle dt-\\
\int_{-\infty}^{+\infty}
\frac{\gamma(t+i\frac{3\varepsilon}2)}{\gamma_A(t+i\frac{\varepsilon}2)}
\left \langle
\frac 1 2\Omega(t+i\frac\varepsilon2) g_+(t+i\frac\varepsilon2),
f_-(t-i\frac\varepsilon2)
\right\rangle dt.
\end{multline}
Due to analytic properties of the functions $g_\pm\in
H_2^\pm(E)$, $f_\pm\in H_2^\pm(E)$ the latter expression has a
limit as $\varepsilon$ tends to $0$ and by Lebesque dominated convergence
theorem and Schwartz inequality
\begin{multline}\label{check this}
\lim_{\varepsilon\downarrow 0}\left\langle \gamma(A+i\varepsilon)\begin{pmatrix}
\tilde g \\ g
\end{pmatrix},\begin{pmatrix}
\tilde f \\ f
\end{pmatrix}\right\rangle=\\
\int_{-\infty}^{-\delta_0} [
\langle \gamma \tilde f, g_-\rangle
+\langle \gamma f, g_+ \rangle
+\frac {\gamma(t)}{{\gamma_A(t)}}\langle \frac 1 2\Omega g_+,f_-\rangle
+\frac {\gamma(t)}{\overline{\gamma_A(t)}} \langle \frac 1 2\Omega' g_-,f_+\rangle
]dt.
\end{multline}
Here $\int [
\langle \gamma \tilde f, g_-\rangle
+\langle \gamma f, g_+ \rangle] dt= \langle \gamma(\tilde g, g), (\tilde f, f)
\rangle$ and therefore represents a meaningful object.

In order to prove that this limit is actually equal to zero, we
recall \cite{NabokoVeselov,VasuninMakarov} that for all $(\tilde g,g) \in H_s(A)$
and for all $(\tilde f, f)\in K$
\begin{equation}\label{test}
\left\langle
[(L-k-i\varepsilon)^{-1}-(L-k+i\varepsilon)^{-1}]\begin{pmatrix}
\tilde g \\ g
\end{pmatrix},\begin{pmatrix}
\tilde f \\ f
\end{pmatrix}
\right\rangle \underset{\varepsilon\to 0}{\longrightarrow} 0
\end{equation}
for a.~a. real $k$. Again taking into account formulae describing
the action of the resolvent of the operator $L$ in the model
representation in upper and lower half-planes, consider the
following expression for arbitrary vectors $(\tilde g,g)\in H_s(A)$, $(\tilde f,f)\in K\equiv H$:
\begin{multline*}
\frac 1{2\pi i}\gamma(t+i\varepsilon)\left\langle
[(A-t-i\varepsilon)^{-1}-(A-t+i\varepsilon)^{-1}]\begin{pmatrix}
\tilde g \\ g
\end{pmatrix},\begin{pmatrix}
\tilde f \\ f
\end{pmatrix}\right\rangle=\\
\frac{\gamma(t+i\varepsilon)}{2 \pi i}\int_{-\infty}^{-\delta_0} \frac {2i \varepsilon}{(k-t)^2+\varepsilon^2}\left\langle \begin{pmatrix}
\tilde g \\ g
\end{pmatrix},\begin{pmatrix}
\tilde f \\ f
\end{pmatrix}\right\rangle dk+\\
\frac{\gamma(t+i\varepsilon)}{{\gamma_A(t+i\varepsilon)}}\left\langle \frac 1 2\Omega(t+i\varepsilon)g_+(t+i\varepsilon),f_-(t-i\varepsilon)
\right\rangle+\\
\frac{\gamma(t+i\varepsilon)}{\overline{\gamma_A(t+i\varepsilon)}}\left\langle \frac 1 2\Omega'(t-i\varepsilon)g_-(t-i\varepsilon),f_+(t+i\varepsilon)
\right\rangle
\end{multline*}
(cf. \eqref{existence2}). The latter expression has a limit for a.~a. $t\in \mathbb
R$, equal to the integrand in \eqref{check this}. On the other
hand, from \eqref{test} it follows, that this limit is identically
equal to zero for a.~a. $t$. This observation completes the proof.

Conversely, let the self-adjoint operator $A$ possess a weak outer
bounded annihilator $\gamma(\lambda)$ in the sense of the Theorem. Let the vector $u\not= 0$, $u\in E(-\infty,0)H$
belong to the absolutely continuous spectral subspace $H_{ac}$.
Then, again by the spectral theorem and by Lebesgue dominated
convergence theorem it is easy to see that
$$
\int_\delta (\gamma(k+i0)-\bar\gamma(k+i0) )d\mu_{u,v} (k)=0
$$
(by taking $E_A(\delta)v$ instead of $v$) for an arbitrary Borel
set $\delta\subset (-\infty,0)$ and the finite absolutely
continuous complex measure \cite{Birman} $d\mu_{u,v}(k):=\langle
dE_A(k)u,v\rangle$, where as above $E_A$ is the operator valued
spectral measure of the operator $A$ and $v$ is an arbitrary
element of $H$. Since boundary values of $\gamma$ are non-zero
almost everywhere on the real line and by assumption these
boundary values are non-real almost everywhere, this implies that
the measure $d\mu_{u,v}\equiv 0$ for all $v\in H$.

This completes the proof.
\end{proof}

\begin{rem}
The last Theorem can of course be easily generalized together with
the proof given to the situation when the set, where one tests the
singularity of the spectrum of the operator $A$, is an arbitrary
finite or infinite interval of the real line or even a finite unit
of such disjoint intervals.
\end{rem}

\begin{rem}
Note that the existence of a non-zero \emph{analytic} bounded
annihilator of the operator $A$ is clearly sufficient for the pure
singularity of its spectrum to the left of the point $\lambda_0$. Nevertheless, our Theorem asserts
that this function can be chosen to be outer in $\mathbb C_+$ as
well.
\end{rem}

\begin{rem}\label{simplicityrem}
Suppose that the operator $A$ is a self-adjoint operator with
simple spectrum. Then the trace class operator $V$ of the last
Theorem due to \eqref{maximalityA} can clearly be chosen
\cite{Birman} as a rank one operator in  Hilbert space $H$. In
this situation, the proof of Theorem \ref{sacase} can be
modified in the part concerning the choice of the annihilator in
the following way: the function $\gamma_A$ can be chosen as
$$
\gamma_A(\lambda):=\frac{1}{1-i(D(\lambda)-1)},
$$
where $D(\lambda):=1+\langle (A-\lambda)^{-1}\phi, \phi\rangle$ is
the perturbation determinant of the pair $A$, $A+\langle \cdot,
\phi\rangle\phi$ and $\phi$ is the generating vector for the
operator $A$.

The proof is a straightforward application of the explicit formula
for the resolvent of a rank one perturbation of a self-adjoint
operator, based on the Hilbert identity.
\end{rem}

We now pass over to the general case, i.e., the case when one has
no a-priori information on the spectral structure of the operator
$A$ near the endpoints of the interval under consideration. In
this case one faces the necessity to modify somewhat the
definition of annihilation. The following Theorem addresses this.

\begin{thm}\label{sacase2}
Let $A$ be a (possibly, unbounded) self-adjoint operator in the
Hilbert space $H$. Let $\Delta$ be an arbitrary Borel set on the
real line. Then the following two statements are equivalent.
\begin{itemize}
\item[(i)] The spectrum of $A$ in $\Delta$ is purely singular,
i.e., the intersection of absolutely continuous spectrum and the
set $\Delta$ is empty; \item[(ii)] There exist an outer bounded in
the upper half-plane non-trivial (i.e., non-zero) scalar function
$\gamma(\lambda)$ and an outer bounded in the upper half-plane
non-constant scalar function $\beta(\lambda)$ such that $\Im
\beta(\lambda)$ has non-tangential limits on the real line at
every point of the latter and these limits are zero everywhere on
$\mathbb R\setminus\Delta$ and non-zero everywhere on $\Delta$,
weakly annihilating the operator $A$ in the following sense:
$$
w-\lim_{\varepsilon\downarrow 0}
\gamma(A+i\varepsilon)(\beta(A+i\varepsilon)-\beta_*(A-i\varepsilon))=0,
$$
where $\beta_*(\lambda):=\bar\beta(\bar\lambda)$ is an outer
bounded in the lower half-plane analytic function.
\end{itemize}
\end{thm}

\begin{proof}
We start with the proof of the implication (i)$\Rightarrow$(ii).

To begin with, let $\beta(\lambda)$ be a Riesz transform of a
square summable non-negative function $b(k)$ such that
$\text{supp}\ b=\Delta$:
\[ \beta(\lambda)=\int \frac{b(k)}{k-\lambda}dk.
\]
In order to satisfy the restrictions of the Theorem on the
imaginary part of $\beta$, further assume that $\beta$ is in
addition a $C^1$ function on the real line. Then clearly it is
outer bounded in the upper half-plane (in fact, even an
R-function), the imaginary part of it has boundary limits
everywhere on $\mathbb R$ \cite{Koosis} and moreover, these
boundary limits are equal to zero on $\mathbb R\setminus \Delta$
and are non-zero everywhere on $\Delta$.

Then $\beta_*(\lambda)$ is an outer bounded analytic continuation
of $\beta$ to the lower half-plane $\mathbb C_-$ through the
complement $\mathbb R\setminus\Delta$, whereas the jump of the
continued function through $\Delta$, which is proportional to
$\Im\beta(k+i0)$, is non-trivial everywhere on $\Delta$.

By the spectral theorem of a self-adjoint operator and then by the
Lebesgue dominated convergence theorem it is now easy to see that
$\beta(A+i\varepsilon)-\beta_*(A-i\varepsilon)\to \beta_0(A)$
strongly as $\varepsilon\to 0$, where $\beta_0(k):=2i\pi\Im
\beta(k+i0)$.

On the other hand, repeating the argument from the proof of the
last Theorem (namely, from \eqref{mainforexistenceA} to
\eqref{check this}, where the integral is extended from
$(-\infty,-\delta_0)$ to the whole real line) one arrives at the
conclusion that $\gamma(\lambda):=\gamma_A(\lambda),$ where
$\gamma_A$ is the same function as above, is such that the
operator family $\gamma(A+i\varepsilon)$ has a weak limit as
$\varepsilon\to 0$, given by \eqref{check this} with the
above-mentioned  change of the limits of integration.

It follows that $w-\lim_{\varepsilon\downarrow 0}
\gamma(A+i\varepsilon)(\beta(A+i\varepsilon)-\beta_*(A-i\varepsilon))$
exists and it's only left to prove that it is equal to zero. Let
first $u\in E_A (\Delta)$. Then $\langle
\gamma(A+i\varepsilon)(\beta(A+i\varepsilon)-\beta_*(A-i\varepsilon))u,v\rangle\to
0$ for all $v\in H$ by the same argument as in the proof of the
preceding Theorem (see \eqref{test} and below).

If on the other hand $u\in E_A(\mathbb R\setminus \Delta)H$, then
the named limit is zero since $\beta_0(A)|_{E_A(\mathbb R\setminus
\Delta)H}=0$ due to the fact that $\Im \beta(k+i0)=0$ for all
$k\in \mathbb R\setminus \Delta$.

The proof of the inverse implication (ii)$\Rightarrow$(i) is
nothing but a slight modification of the corresponding implication
of Theorem \ref{sacase}

Indeed, let the vector $u\not= 0$, $u\in E_A(\Delta)H$ belong to
the absolutely continuous spectral subspace $H_{ac}$. Then, again
by the spectral theorem and by Lebesgue dominated convergence
theorem it is easy to see that
$$
\int_\delta \gamma(k+i0)(\beta(k+i0)-\bar\beta(k+i0) )d\mu_{u,v}
(k)=0
$$
(by taking $E_A(\delta)v$ instead of $v$) for an arbitrary Borel
set $\delta\subset \delta$. Since boundary values of $\gamma$ are
non-zero almost everywhere on the real line and by assumption the
boundary values of the imaginary part of $\beta$ are non-zero in
$\Delta$, this implies that the absolutely continuous measure
$d\mu_{u,v}\equiv 0$ for all $v\in H$.

This completes the proof.

\end{proof}

\specialsection{On the analytic properties of the resolvent}

We take this opportunity to prove yet another result. We begin with the following observation, well-known from the mathematical scattering theory. Consider
a self-adjoint operator $A$. Then there exists a linear set $\tilde H_{a.c.}$ dense in the absolutely continuous spectral subspace of $A$ such that
$$
\int \|\beta \exp (iAt)u\|^2dt<\infty
$$
for all $u\in \tilde H_{a.c.}$ and any non-negative operator $\beta\in\mathfrak S_2$ (see, e.g., \cite{Yafaev}). Using the Fourier transform and Parseval's
identity,
it's easy to see \cite{MIAN} that the last condition is equivalent to:
$$
\beta (A-\lambda)^{-1}u\in H^2_\pm(Ran\ \beta)
$$
for all $u\in \tilde H_{a.c.}$ Taking an operator $V\in \mathfrak
S_1$ as in the proof of the previous Theorem, i.e., a non-negative
trace class operator such that the condition \eqref{maximalityA}
is satisfied, we can further obtain \cite{MIAN} the following
description of the absolutely continuous spectral subspace of the
operator $A$:
$$
H_{a.c.}=clos \{u| \sqrt{V}(A-\lambda)^{-1}u\in H^2_\pm (E)\},
$$
where as in Section 2 $E$ is the auxiliary Hilbert space, being the closed image of the operator $V$.

In this Section, we derive an analogous characterization for the singular spectral subspace $H_s$ of a self-adjoint operator $A$. Namely, the following
Theorem holds.
\begin{thm}\label{strongdescription}
Let $A$ be a self-adjoint operator in the Hilbert space $H$. Let $V\in\mathfrak S_1$ be a positive trace class operator in $H$ such that \eqref{maximalityA}
holds. Then if the vector $u$ belongs to the singular spectral subspace $H_s$ of $A$, then
the vector $\sqrt{V} (A-\lambda)^{-1}u$ belongs to
vector Smirnov classes $N^2_\pm(E)$ \cite{Nik1}, i.~e., it can be
represented as $h_\pm(\lambda)/\delta_\pm(\lambda)$, where $h_\pm\in
H^2_\pm(E)$ and $\delta_\pm(\lambda)$ are scalar bounded outer analytic
functions in half-planes $\mathbb C_\pm$, respectively. Here the functions
$\delta_\pm$ can be chosen independently of vector $u$.
\end{thm}

\begin{proof}
We again use the functional model constructed based on the dissipative operator $A+iV$.

Let now $u\in H_{s}$. The following identities hold (see \cite{MIAN}):
\begin{equation}\label{pereschet}
\begin{gathered}
\sqrt{2 \pi} g_+(\lambda) =-\Theta_A(\lambda) \alpha
(A-\lambda)^{-1}u,\quad Im\ \lambda>0,
\\
\sqrt{2 \pi} g_-(\lambda) =-\Theta_A'(\lambda) \alpha
(A-\lambda)^{-1}u,\quad Im\ \lambda<0
\end{gathered}
\end{equation}
Here the operator-functions $\Theta_A(\lambda)$ and $\Theta_A'(\lambda)$ are defined by the identities
$\Theta_A(\lambda)=(I+S(\lambda))/2$ and $\Theta_A'(\lambda)=
(I+S^*(\overline\lambda))/2$ ($S$ being the characteristic function of the dissipative operator $A+iV$), and are outer $\mathfrak S_1$--valued
contractions in the half-planes $\mathbb C_+$ and $\mathbb C_-$,
respectively.

Within the conditions of Theorem
\ref{strongdescription} both operator-functions $\Theta_A(\lambda)$ and $\Theta_A'(\lambda)$ also possess outer determinants
in their respective half-planes
\cite{Nagy}. Therefore, by the uniqueness theorem for scalar bounded
analytic functions \cite{Hoffman}, they are invertible
for almost all real $k$.

Then we obtain immediately,
that for all $\lambda\in \mathbb C_+$
$$
\sqrt{2V} (A-\lambda)^{-1} u=-\sqrt{ 2 \pi}
\Theta_A^{-1}(\lambda)g_+(\lambda)= -\sqrt{2 \pi}
\delta_+^{-1}(\lambda) \Omega(\lambda)g_+(\lambda),
$$
where
$\Omega(\lambda)\Theta_A(\lambda)=\Theta_A(\lambda)\Omega(\lambda)=\delta_+(\lambda)
I$ with a bounded operator-function $\Omega(\lambda)$, i.e.,
$\delta_+$ is the determinant of the operator function
$\Theta_A$. It remains to point out (see Section 2), that the function $g_+(\lambda)$ belongs
to $H^2_+(E)$ as $u\in H$. Application of a similar
argument to the vector $g_-$ completes the proof
\end{proof}

\begin{rem}
As it is easily seen from the definition of $\Theta_A(\lambda)$ and $\Theta_A'(\lambda)$, the functions $\delta_\pm(\lambda)$ in the statement
of Theorem \ref{strongdescription} can be chosen so that $\delta_-(\lambda)=\bar\delta_+(\bar\lambda)$.
\end{rem}

In a similar way we are able to give a ``weak'' version of the previous Theorem. Indeed, one can easily ascertain
(see, e.g., \cite{Yafaev,MIAN}) on the basis of F. and M. Riesz theorem  \cite{Hoffman} that
the absolutely continuous subspace of a self-adjoint operator $A$ can be alternatively characterized as follows:
$$
H_{a.c.}=clos \{u| \langle(A-\lambda)^{-1}u,v\rangle\in H^2_\pm\}\quad \text{for all } v\in H.
$$
The following Theorem gives an analogous representation for the singular spectral subspace.
\begin{thm}\label{weakdescription}
Let $A$ be a self-adjoint operator in the Hilbert space $H$.
Then if the vector $u\in H$ belongs to the singular spectral subspace
$H_s$, then
the function $\langle (A-\lambda)^{-1}u, v \rangle$ belongs to
Smirnov classes $N^1_\pm$ for all $v\in H$, i.~e., it can be
represented as $h_\pm(\lambda)/\delta_\pm(\lambda)$, where $h_\pm\in
H^1_\pm$ and $\delta_\pm(\lambda)$ are bounded scalar outer analytic
functions in half-planes $\mathbb C_\pm$, respectively.
Here the functions
$\delta_\pm$ are independent of $v\in H$ and can be chosen  independently of vector $u$.
\end{thm}

\begin{proof}
We will again use the model description of the resolvent of the operator $A$ \eqref{samodel}, from where it follows that
$$
\langle (A-\lambda)^{-1}u,v\rangle=
\left\langle \frac 1{k-\lambda}\binom{\tilde g}{g},\binom{\tilde f}{f}\right\rangle
-\left\langle \frac 1{k-\lambda} \binom{\frac 12\Theta_A^{-1}(\lambda)g_+(\lambda)}{0},\binom{\tilde
f}{f}\right\rangle,
$$
where $(\tilde f,f)$ is the model image of the vector $v$. Let $u\in H_s$. The first term on the right hand side is
clearly the Cauchy transformation of an $L_1$-function, whereas the second one can be rewritten by
residue calculation in
the following way:
$$
\left\langle \frac 1{k-\lambda} \binom{\frac 12 \Theta_A^{-1}(\lambda)g_+(\lambda)}{0},\binom{\tilde
f}{f}\right\rangle =-2 \pi i \langle \frac 12 \Theta_A^{-1}(\lambda) g_+(\lambda),
f_-(\overline{\lambda})\rangle_E.
$$
By Theorem \ref{strongdescription} the vector $u$ is such that $\sqrt{V} (A-\lambda)^{-1}u\in N^2_+(E)$,
and therefore by \eqref{pereschet} again, $\Theta_1^{-1} (\lambda)
g_+(\lambda)=h_+(\lambda)/\delta_+(\lambda)$
for some $h_+\in H^2_+(E)$ and some outer bounded in
the upper half-plane function $\delta_+$.
It follows that if one puts $\nu(\lambda):=1/(\lambda+i)$,
\begin{multline*}
\langle (A-\lambda)^{-1}u,v\rangle
= k_1(\lambda) -2\pi i \frac 1 {\delta_+(\lambda)} \langle \frac 12 h_+(\lambda),
f_-(\overline{\lambda})\rangle_E\equiv\\
\frac 1{\delta_+(\lambda)\nu(\lambda)}[k_1(\lambda)\delta_+(\lambda)\nu(\lambda) -2\pi i  \langle \frac 12 h_+(\lambda),
f_-(\overline{\lambda})\rangle_E \nu(\lambda)]\in N^1_+,
\end{multline*}
since $k_1(\lambda):= \langle \frac 1{k-\lambda}\binom{\tilde g}{g},\binom{\tilde f}{f}\rangle
$ and $\overline{f_-(\overline{\lambda})}\in H^2_+(E)$.

An analogous argument applied in the case of $\mathbb C_-$ completes the proof.
\end{proof}
\begin{rem}
Note that the functions $\delta_\pm(\lambda)$ appearing in the proof of the last Theorem are the same as in the proof of Theorem
\ref{strongdescription}, i.e., these can be chosen to be equal to the determinants of the operator-functions $\Theta_A(\lambda)$ and
$\Theta_A'(\lambda)$, respectively. Thus, the corresponding outer factors in Theorems \ref{strongdescription} and \ref{weakdescription} admit
the simplest (and explicitly computable) form in the situation when the spectrum of the operator $A$ is simple, see Remark \ref{simplicityrem}.
\end{rem}

\subsection*{Acknowledgements.}

Both authors express their gratitude to Prof. David Pearson for
the interest expressed by him to their research and for the
question that motivated this paper.

The first author is grateful to Prof. A. Sobolev for fruitful discussions during the author's stay in UCL.

The first author is grateful to the Dept. of Mathematics,
University College London where parts of this work were done for
hospitality.

%%% ----------------------------------------------------------------------
\end{document}